\documentclass{article}
\usepackage{amsmath}
\usepackage{amsthm}
\usepackage{amsfonts}
\usepackage{tcolorbox}
\usepackage{graphicx}

\usepackage{authblk}

\title{Graph bundles and  Ricci-flatness}

\author[1]{Wenbo Li}\author[2]{Shiping Liu}
\affil[1,2]{School of Mathematical Sciences, University of Science and Technology of China, Hefei 230026, China}
\affil[1]{patlee@mail.ustc.edu.cn}
\affil[2]{spliu@ustc.edu.cn}
\date{}
\begin{document}
\maketitle

\newtheorem{theorem}{Theorem}
\newtheorem{proposition}{Proposition}
\newtheorem{corollary}{Corollary}
\newtheorem{lemma}{Lemma}
\newtheorem{definition}{Definition}
\newtheorem{remark}{Remark}
\newtheorem{eg}{Example}

\begin{abstract}
     We develop a systematic way of constructing S-Ricci flat graphs which are not Abelian Cayley via graph bundle with explicit examples. For this purpose, we prove that, with some natural constrains,  a non-trivial graph bundle cannot be isomorphic (as graphs) to the product of the base graph and fiber graph. It stands in clear contrast to the continuous case.
     \\

     
\end{abstract}

\section{Introduction}
Let $F \hookrightarrow E \to B$ be a fiber bundle with fiber $F$, total space $E$ and base $B$. It is known that even if the bundle structure $F\hookrightarrow E \to B$ is non-trivial the total space $E$ can still be homeomorphic to the product space $B \times F$, even in the case that $E$ is a vector bundle. The following is an example. Consider the canonical projection $$\pi: S^{2} \times \mathbf{R} \to  S^{2} $$ and let $\xi$ be the pull-back of the tangent bundle $TS^{2}$ under $\pi$. One can regard $\xi$ as the product space $TS^{2} \times \mathbf{R}$. Now we consider an inclusion map $$i: S^{2} \hookrightarrow  S^{2} \times \mathbf{R}$$ such that $\pi \circ i=i $. Observe that the pull-back of $\xi$ under $i$ is the tangent bundle $TS^{2}$
, which is non-trivial by the Hairy Ball Theorem, thus the bundle structure $\mathbf{R}^{2} \hookrightarrow \xi \to S^{2} \times \mathbf{R}$ is also non-trivial. However, we have the canonical homeomorphism:
$$TS^{2} \times \mathbf{R} \cong S^{2} \times \mathbf{R}^{3}.$$ As a result, there exists a homeomorphism
$$\xi \cong (S^{2} \times \mathbf{R}) \times \mathbf{R}^{2}$$ between the total space and the product space of the base and fiber.

The discrete analogy of fiber bundles, which is called \emph{graph bundles} following \cite{pisanski1983edge,mohar1988maximum,kwak1990isomorphism}, is defined and studied as a generalization of both coverings and products of graphs, see, e.g.,  \cite{sohn1994characteristic,chae1993characterlstlc,imrich1997recognizing,hong1999bipartite,klavvzar1995coloring,klavzar1995chromatic,kwak1996isoperimetric,zmazek2000recognizing,zmazek2002algorithm,zmazek2002unique,zmazek2006domination,vzerovnik2000recognition,banivc2006fault,banivct2010wide,banivct2009edge,feng2006zeta,kwak2001characteristic,pisanski2009hamilton,ervevs2013mixed,kim2008generalized}. Motivated by the example above, a natural question is whether a non-trivial graph bundle could be isomorphic (as graphs) to the product of the base graph and fiber graph. 

In contrast to the continuous case, we show that with some natural constraints the answer is no (see Theorem \ref{thm:2}). This indicates a special discrete nature of graph bundles. Moreover, we introduce a new concept of \emph{discrete vector bundle} (see Definition \ref{def:2}), which is a graph bundle with a null section. 
We prove that a nontrivial discrete vector bundle with both base graph and fiber graph being vertex-transitive fails to be vertex-transitive itself (see also Theorem \ref{thm:2}).
The requirement that the graph bundle being a discrete vector bundle here is necessary (see Example \ref{eg:dvb1}).

As an application of the above results, we develop a systematic way of constructing \emph{Ricci flat} graphs via graph bundles. In particular, we explicitly construct a family of \emph{S-Ricci flat} graphs which are not Abelian Cayley (see Example \ref{eg2}).  

The concept of Ricci flat graphs has been introduced  by Chung and Yau \cite{chung1996logarithmic}, as a discrete analogue of Riemannian manifolds with vanishing Ricci curvature. During these years the study of various notions of discrete Ricci curvature on graphs attracts a lot of attentions, see, e.g., 
%\cite{BE1985,CKKLMP2020Adv,
%CKLP2022,CLP2020, Hua2019Liouville, HuaLin2017Stochastic, JostLiu2014, KKRT2016Cana, LLY2011Tohoku,LMP2018Cacl, LMP2022MathAnn,LinYau2010,Muench2023JMPA,MW2019Adv,Ollivier2009JFA, Salez2022GAFA, Schmuckenschlaeger1999,Siconolfi2021}, and 
the survey papers \cite{BHJLW17survey,HuaLin2016survey,kamtue2018curvature,Maas17survey} and the references therein. In \cite{cushing2021curvatures}, Cushing et al. relates two important Ricci curvature notions on graphs, namely the Ollivier/Lin-Lu-Yau curvature \cite{Ollivier2009JFA,LLY2011Tohoku,MW2019Adv} and Bakry-\'Emery curvature \cite{BE1985,Elworthy1991,LinYau2010,Schmuckenschlaeger1999}, via the concept of Ricci-flatness. In particular, we mention that any Ricci-flat graph has nonnegative curvature in the sense of both Bakry-\'Emery and Ollivier/Lin-Lu-Yau.

Recall that a $d$-regular simple graph $G=(V,E)$ is called \emph{Ricci flat} if for any $x \in V$, there exist maps \[\eta_{i}:B_{1}(x)\rightarrow V, \,\, 1\leq i\leq d,\] where $B_1(x):=\{x\}\cup \{y\in V: y\sim x\}$, such that
\begin{itemize}
    \item [(i)]$\eta_{i}(u) \sim u$, $u \in B_{1}(x)$,
    \item [(ii)] $\eta_{i}(u)=\eta_{j}(u)$ iff $i=j$,
    \item [(iii)]$\bigcup_{j}\eta_{i}(\eta_{j}(x))=\bigcup_{i}\eta_{i}(\eta_{j}(x))$.
\end{itemize}
And it is called \emph{S-Ricci flat} \cite{cushing2021curvatures} or \emph{locally Abelian} if furthermore 
\begin{itemize}
    \item [(iv)] $\eta_{i}(\eta_{j}(x))=\eta_{j}(\eta_{i}(x))$.
\end{itemize}

The concept is motivated by lattices and, in general, Abelian Cayley graphs. The property of being S-Ricci flat has been employed in characterizing vertices with vanishing Ollivier curvature under certain constraints \cite[Theorem 4.5 (a)]{cushing2021curvatures}.
One may wonder how different the class of S-Ricci flat graphs and the class of Abelian Cayley graphs could be. We address this question using the tool of graph bundles. Our approach showes that S-Ricci flat graphs are not always vertex transitive. Moreover, the number of orbits in automorphism groups of S-Ricci flat graphs can be arbitrarily large (see Example \ref{eg3}).

\section{Graph bundles and discrete vector bundles}
\begin{definition}[\textbf{Graph bundle}]
Let $G=\{V_{G},E_{G}\}$ and $F=\{V_{F},E_{F}\}$ be two locally finite, simple, connected graphs with $$E^{or}_G=\{(x,y); x,y \in V_{G},x \sim y\}$$ being the set of oriented edges of $G$. Let $\phi$ be a map \[\phi:E_G^{or}\to Aut(F): (x,y)\mapsto \phi_{yx},\] to the automorphic group $Aut(F)$ of the graph $F$ such that \[\phi_{xy}=\phi_{yx}^{-1}.\] 
Then we define the graph bundle $G\times_{\phi}F=\{V,E\}$ such that $V=V_{G}\times V_{F}$, and $E$ given as follows: For any $(x,v),(y,w)\in V$, $(x,v)\sim (y,w)$ if and only if one of the following holds:
\begin{itemize}
    \item [(i)] $x=y$, $v \sim w$ in F;
    \item [(ii)] $x \sim y $ in G, $w=\phi_{yx}(v)$.
\end{itemize}
We call $G$, $F$ and $\phi$ the base graph, fiber graph and connection, respectively. The graph bundle $G\times_{\phi}F$ is also referred as an $F$-bundle over $G$. 
\end{definition}

Notice that the natural projection map $\pi$ from $V=V_G\times V_F$ to $V_{G}$ has the following property: for any edge $u \sim v$ in $E$, we have either $\pi(u)=\pi(v)$ or $\pi(u)\sim \pi(v)$ in $E_G$. 

Recall from the definition of a continuous vector bundle there is always a null section. Motivated by this fact, we propose the following definition for later use.

\begin{definition}[\textbf{Discrete vector bundle}]\label{def:2}
A graph bundle $G\times_{\phi}F$ is called a discrete vector bundle if there exists a vertex $o \in V_F$ such that every automorphism $\phi_{yx}$ with $(x,y)\in E^{or}_G $ fix $o$. The vertex $o$ is called the null element in $F$ and the graph embedding $i: G\to G\times_{\phi}F, x \mapsto (x,o)$ is called a null section.
\end{definition}

\noindent The following definition is taken from the vector bundle theory.

\begin{definition}[\textbf{Equivalence of bundles}]
Consider two graph bundles $G\times_{\phi_1} F$ and $G\times_{\phi_2} F$.
The two connections $\phi_{1}$ and $\phi_{2}$ are equivalent if there exists an graph isomorphism $\varphi$ between $G\times_{\phi_{1}}F$ and $G\times_{\phi_{2}}F$ satisfying $\pi \circ \varphi=\pi$.
\end{definition}

\begin{remark}
\rm{We point out that two inequivalent graph bundles might still be isomorphic as graphs.}
\end{remark}
\begin{remark}
\rm{
For the connection $\phi$ that assigns to each oriented edge of $G$ the identity map in $Aut(F)$, the graph bundle $G\times_{\phi}F$ is simply the Cartesian product of $G$ and $F$. Any connection equivalent to it is called \textbf{trivial}. For our purpose here,  we need to rule out the possibility of a non-trivial bundle having a Cartesian product structure. We will settle this problem in the next section, with more assumptions on the base and fiber graphs.}
\end{remark}

\noindent Here is another definition needed in this article.
\begin{definition}[\textbf{Paths}]
For any graph $G=\{V,E\}$, a path of length $N \geq 1$ in $G$ is a finite sequence $\{x_{i}\}_{i=0}^N$ with $x_{i} \in V$ and $x_{i} \sim x_{i+1}$ for $0\leq i \leq N-1$. If, in addition, $x_{0}=x_{N}$, we call it a loop. Sequences with only one element are also regarded as loops of length 0 as well. The length of a path $\gamma$ is denoted as $|\gamma|$.
\end{definition}

Notice that we allow $x_i=x_j$ for $|i-j|>1$ in a path $\{x_{i}\}_{i=0}^N$.

We are now prepared to present our first result. 

\begin{theorem} Let $G\times_\phi F$ be a graph bundle.
The connection $\phi$ is trivial if and only if for any loop $\{x_{i}\}_{i=0}^N$ in $G$,  \begin{equation}\label{eq:1}
\prod_{l=0}^{N-1}\phi_{x_{l+1},x_{l}}=id\in Aut(F),
\end{equation}
where the product is defined by multiplication on the left.
\end{theorem}

\begin{proof}
Assume that the condition (\ref{eq:1}) holds. For any given $x\in V_G$, we define the map
\[\rho: V_G\to Aut(F): y\mapsto \rho_y \] as follows. 
For any $y\neq x$, pick any path $\gamma=\{x_{i}\}_{i=0}^N$ connecting $x$ and $y$ and define $$\rho_{y}:=\prod_{l=0}^{N-1}\phi_{x_{l+1},x_{l}}.$$
Set $\rho_{x}=id\in Aut(F)$. Due to the fact (\ref{eq:1}), the definition of $\rho_{y}$ does not rely on the choices of the path. That is, the map $\rho$ is well-defined. 

Apparently, the map $\psi: G \times F\rightarrow  G\times_{\phi}F$, $\psi(y,v)=(y,\rho_{y}(v))$ satisfies $\pi \circ \psi=\pi$.
It remains to show that $\psi$ is a graph isomorphism. It is direct to check that $\psi$ is a bijection between vertex sets.  Moreover, for $v \sim w$ in $F$, it holds  $$\psi(y,v)=(y,\rho_{y}(v))\sim (y,\rho_{y}(w))=\psi(y,w).$$  For $y_{1}\sim y_{2}$ in $G$, we have $$\psi(y_{1},v)=(y_{1},\rho_{y_{1}}(v)) \sim (y_{2},\rho_{y_{2}}(v))=\psi(y_{2},v)$$ since it holds by the definition of $\rho$ that $\rho_{y_{2}}(v)=\phi_{y_{2},y_{1}}(\rho_{y_{1}}(v))$. This proves that $\psi$ is a graph isomorphism. Hence, the connection $\phi$ is trivial.

Next, we assume that the connection $\phi$ is trivial, i.e., there exists a graph isomorphism $\varphi$ from $G\times_{\phi}F$ to $G\times F$ with $\pi \circ \varphi=\pi$. For any loop $\{x_{i}\}_{i=0}^N$ in $G$ and any given $v\in V_F$, $\{(x_{i},\prod_{l=0}^{i-1}\phi_{x_{l+1},x_{l}}(v))\}_{i=0}^N$ is a path in $G\times_{\phi}F$.
Since $\pi\circ\varphi=\pi$, we have for any $0\leq i\leq N$, \[\varphi((x_i, \prod_{l=0}^{i-1}\phi_{x_{l+1},x_{l}}(v)))=(x_i,y_i),\]
for some $y_i\in V_F$. Noticing that $\varphi$ is a graph isomorphism, we derive
 $$(x_j,y_j) \sim (x_{j+1},y_{j+1})\,\,\text{for}\,\,0 \leq j \leq N.$$
 Since $x_j\neq x_{j+1}$, we have $y_{j}=y_{j+1}$ for $0 \leq j \leq N $, especially $y_{0}=y_{N}$. Thus $\{(x_i,y_i)\}_{i=0}^N$ is a loop and so does its preimage  $\{(x_{i},\prod_{l=0}^{i-1}\phi_{x_{l+1},x_{l}}(v))\}_{i=1}^N$. We conclude that $v=\prod_{l=0}^{N-1}\phi_{x_{l+1},x_{l}}(v)$. Since $v$ is arbitrarily chosen, we obtain (\ref{eq:1}).
\end{proof}

Note that any connection on a tree is trivial. A loop in the base graph of a graph bundle is called \textbf{balanced} (resp. \textbf{unbalanced}) if the condition (\ref{eq:1}) does (resp. does not) hold.

\section{Graph bundles and vertex transitivity}

Recall that (connected) differentiable manifolds are homogeneous. That is, for any two given points $p,q$ in a connected differentiable manifold $M$, there exists a diffeomorphism from $M$ to itself mapping $p$ to $q$. Graphs with a similar property are called vertex transitive graphs, i.e., for any pair of vertices, there exists a graph isomorphism that maps one to the other. In this section, we prove the following theorem.

\begin{theorem}\label{thm:2}
Let $G$ and $F$ be two vertex-transitive graphs and $\phi$ be an associated non-trivial connection. Then the graph bundle $G\times_{\phi}F$ is not isomorphic (as graphs) to $G\times F$. Moreover, if $G\times_{\phi}F$ is a discrete vector bundle, then it cannot be vertex-transitive.
\end{theorem}

\begin{remark}
    Generally a non-trivial graph bundle $G\times_{\phi}F$ might still be vertex transitive. See Section \ref{sec:4}, Example \ref{eg:dvb1} and \ref{sec:5}, Example \ref{eg:dvb2}.
\end{remark}

To prove Theorem \ref{thm:2}, we study projections of paths. We first explain the concept in the case of Cartesian products of graphs. 
\begin{definition}[\textbf{Projections of a path in a Cartesian product of graphs}]\label{def:projection_Product}
     For any given path $\gamma= \{(x_{i},v_{i})\}_{i=0}^N$ in $G\times F$, the projection $\pi_1(\gamma):=\{p_i\}$ of $\gamma$ to the graph $G$ is a path defined inductively as below: For any $i$, set $p_i=x_{k(i)}$, with $k(0)=0$, and \[k(i+1)=\min\{n\in \mathbb{N}: n\geq k(i), x_n\neq x_{k(i)}\}.\]
     In the same way we can define the projection $\pi_2(\gamma)$ of $\gamma$ to the graph $F$.
\end{definition}
 Notice that a loop in $G\times F$ of length N is projected to two shorter loops $\pi_1(\gamma)$ in $G$ and $\pi_2(\gamma)$ in $F$, respectively, with $|\pi_1(\gamma)|+|\pi_2(\gamma)|=N$. 
 
 We have the following lemma.
\begin{lemma}\label{Lemma 1}
For any loop $\gamma_{1}$ in $G$ starting at $x \in G$ with $|\gamma_{1}|=m$ and $\gamma_{2}$ in $F$ starting at $v \in F$ with $|\gamma_{2}|=n$, there are exactly $\tbinom{m+n}{n}=\frac{(m+n)!}{m!n!}$ loops starting at $(x,v)$ in $G \times F$ with $\gamma_{1}$ and $\gamma_{2}$ as their projections.
\end{lemma}

We omit the proof, which is straightforward. We will only need the fact that the number of loops in $G\times F$ with $\gamma_1$ and $\gamma_2$ as projections depends only on the lengths $|\gamma_1|$ and $|\gamma_2|$.

 For general graph bundles, the projection of a path towards the base graph can be defined similarly as in Definition \ref{def:projection_Product}, while the projection to the fiber graph need to be carefully addressed. 
 
To ease our notations, we denote
$
\phi_{xx}=id \in Aut(F)
$
for every $x\in G$.

\begin{definition}[\textbf{Projections of a path in a graph bundle}]
    For any given path $\gamma= \{(x_{i},v_{i})\}_{i=0}^{N}$ in $G\times_{\phi}F$, the projection $\pi_1(\gamma):=\{p_i\}$ of $\gamma$ to the graph $G$ is defined inductively as follows: For any $i$, set $p_i=x_{k(i)}$, with $k(0)=0$, and \[k(i+1)=\min\{n\in \mathbb{N}: n\geq k(i), x_n\neq x_{k(i)}\}.\]
    The projection $\pi_2(\gamma):=\{q_i\}$ of the path $\gamma$ to the fiber $F$ is defined as follows: First set $\{w_i\}_{i=0}^{N}$ a sequence in $F$ such that $w_{0}=v_{0}$ and 
    $$w_{i}=\left(\prod_{j=0}^{i-1}\phi_{x_{j+1},x_{j}}\right)^{-1}(v_{i}), i\geq 1.$$ 
    The projection $\pi_2(\gamma)=\{q_i\}$ is defined inductively as $q_i=w_{s(i)}$, with $s(0)=0$, and \[s(i+1)=\min\{n\in \mathbb{N}: n\geq s(i), w_n\neq w_{s(i)}\}.\]

\end{definition}

 In the graph bundle $G\times_{\phi} F$ a loop starting at $(x,v)$ of length $N$ is projected into a loop $\pi_1(\gamma)$ in $G$ and a path $\pi_2(\gamma)$ in $F$ connecting $v$ and $(\prod_{l=0}^{N-1}\phi_{x_{l+1},x_{l}})^{-1}(v) $ with $|\pi_1(\gamma)|+|\pi_2(\gamma)|=N$. Notice that in general the path $\pi_2(\gamma)$ might not be a loop. 
 
 The following analogue to Lemma \ref{Lemma 1} holds.

\begin{lemma} \label{lemma 2}
For any loop $\gamma_{1}=\{x_{i}\}$ in $G$ starting at $x \in G$ with $|\gamma_{1}|=m$ and any path $\gamma_{2}$ in $F$ connecting $v \in F$ and $(\prod_{l=0}^{m-1}\phi_{x_{l+1},x_{l}})^{-1}(v)$ with $|\gamma_{2}|=n$, there are exactly $\tbinom{m+n}{n}=\frac{(m+n)!}{m!n!}$ paths starting at $(x,v)$ in $G \times_{\phi} F$ with $\gamma_{1}$ and $\gamma_{2}$ being its projections.
\end{lemma}

Again we omit the proof. Next we prove Theorem \ref{thm:2}:

\begin{proof}
We prove the first part of Theorem \ref{thm:2} by contradiction. Suppose there exists an isomorphism between a non-trivial bundle $G \times_{\phi} F$ and the graph product $G \times F$. For every $s \in V_G$, denote by $m(s)$ the length of the shortest unbalanced loop(s) in $G$ based at $s$. Find $x_0 \in V_G$ such that 
\[m(x_0) \leq m(y),\ \text{for any}\ y \in V_G.\]
Clearly $m(x_{0}) \geq 3$. By definition of an unbalanced loop, there exists $v_0 \in V_F$ and a loop $\{x_{i}\}_{i=0}^{m(x_{0})} $ based at $x_{0} \in G$ such that 
$$\prod_{l=0}^{m(x_0)-1}\phi_{x_{l+1},x_{l}}(v_{0}) \neq v_{0}.$$ 
 By assumption, $G$ and $F$ thus $G \times F$ are all vertex transitive. Therefore, there exists an isomorphism between $G \times_{\phi} F$ and $G \times F$ which maps $(x_0,v_0) \in V_{G \times_{\phi} F}$ to $(x_0,v_0) \in V_{G \times F}$.

Let $\Gamma_{1}$ (resp. $\Gamma_{2}$) be the set of all loops in $G \times_{\phi} F$ (resp. $G \times F$) that are based at $(x_0,v_0)$ of length $m(x_{0})$. To get a contradiction we will show $$\#\Gamma_{1}<\#\Gamma_{2}.$$ This is done in the following two steps.

\noindent \textbf{STEP 1} First we prove 
$$\#\{\gamma:\gamma \in \Gamma_{1}, |\pi_1(\gamma)|< m(x_0)\} =\#\{\gamma:\gamma \in \Gamma_{2}, |\pi_1(\gamma)|< m(x_0)\},$$
where $\pi_1(\gamma)$ is the projection of $\gamma$ to the base graph $G$. 
 For any
$$\gamma= \{((x_{i},v_i)\}_{i=0}^{m(x_{0})}\in \{\gamma:\gamma \in \Gamma_{1}, |\pi_1(\gamma)|< m(x_0)\},$$
the projection  $\pi_1(\gamma)$ must be a balanced loop due to the choice of $m_{x_0}$.  This implies that $$\prod_{l=0}^{m(x_0)-1}\phi_{x_{l+1},x_{l}}(v_{0})= v_{0},$$
and hence the projection $\pi_2(\gamma)$ of $\gamma$ to the fiber $F$ is a loop. That is, every $\gamma \in \Gamma_1$ with $|\pi_1(\gamma)|< m(x_0)$ is projected to two loops in $G$ and $F$, respectively. 

Then, we calculate $\#\{\gamma:\gamma \in \Gamma_{1}, |\pi_1(\gamma)|< m(x_0)\}$ in the following way. 

\textit{Pick any $\gamma_1$ based at $x_0$ in $G$
with $|\gamma_1|<m(x_0)$
and any $\gamma_2$ based at $v_0$ in $F$ such that $|\pi_1(\gamma)|+|\pi_2(\gamma)|=m(x_0)$. Count by Lemma \ref{lemma 2} the number of loops with $\gamma_1$ and $\gamma_2$ as its projections to the base and fiber, respectively. This number depends only on the length $|\gamma_1|$ and $m(x_0)$. Then $\#\{\gamma:\gamma \in \Gamma_{1}, |\pi_1(\gamma)|< m(x_0)\}$ is derived by summing up all these numbers over every pair of such $\gamma_1$ and $\gamma_2$.}

Notice that $\#\{\gamma:\gamma \in \Gamma_{1}, |\pi_1(\gamma)|< m(x_0)\}$ can be calculated in the same way with the same result by employing Lemma \ref{Lemma 1}. This completes the proof of Step 1.

\noindent \textbf{STEP 2} Secondly we prove that $$\#\{\gamma:\gamma \in \Gamma_{1}, |\pi_1(\gamma)|= m(x_0)\} <\#\{\gamma:\gamma \in \Gamma_{2}, |\pi_1(\gamma)|= m(x_0)\}.$$
Let us denote $L:=\{\text{loops based at }\,x_0\,\,\text{in}\,\,G\,\,\text{of length}\,\,m(x_0)\}$.
Clearly the right hand side equals the cardinality of $L$. 
For the left hand side, we observe that the projection map  $$\pi: \{\gamma:\gamma \in \Gamma_{1}, |\pi_1(\gamma)|= m(x_0)\}\to L$$ is injective. However, by our assumption, there exists a loop $\{x_i\}_{i=0}^{m(x_0)}$ in $L$ such that
$$\prod_{l=0}^{m(x_0)-1}\phi_{x_{l+1},x_{l}}(v_{0})\neq v_{0}.$$ 
Hence, there exists no $\gamma\in \{\gamma:\gamma \in \Gamma_{1}, |\pi_1(\gamma)|= m(x_0)\}$ with $\pi(\gamma)=\{x_i\}_{i=0}^{m(x_0)}$. Indeed, if we have such a $\gamma$, then $\pi_2(\gamma)$ is a path connecting $(\prod_{l=0}^{m(x_0)-1}\phi_{x_{l+1},x_{l}})^{-1}(v_{0})$ and $ v_{0}$ which has a positive length. This yields that $$|\gamma|=|\pi_1(\gamma)|+|\pi_2(\gamma)|>|\pi_1(\gamma)|=m(x_0),$$ which is a contradiction. 

In conclusion, the map $\pi$ is injective but not surjective. Therefore, we have $$\#\{\gamma:\gamma \in \Gamma_{1}, |\pi_1(\gamma)|= m(x_0)\}<\# L=\#\{\gamma:\gamma \in \Gamma_{2}, |\pi_1(\gamma)|= m(x_0)\}.$$

Combining Step 1 and Step 2, we prove the first part of Theorem \ref{thm:2}. 

For the second part of Theorem \ref{thm:2}, we choose $(x_0,v_0)$ as in the proof of the first part of the theorem. Let $m(x_0)$ still be the same as above. By the same method, one show that the numbers of loops of length $m(x_0)$ based at $(x_0,v_0)$ and $(x_0,o)$, respectively, are different, where $o \in F$ is the null element. Since such number is invariant under the graph isomorphism, there is no isomorphism which maps $(x_0,v_0)$ to $(x_0,o)$. Hence, we conclude that $G\times_{\phi}F$ is not vertex transitive.
\end{proof}

\section{Examples of non-trivial graph bundles}\label{sec:4}

In this section, we construct several families of explicit graph bundles. Recall from graph theory the following definition.

\begin{definition}[\textbf{Cayley Graph}]
Let $G$ be a group and $\Gamma \subset G$ generate $G$. A Cayley graph $(G,\Gamma)$ is a graph with vextex set $G$ and for $g_1,g_2 \in G$, $g_{1}\sim g_{2}$ if $g_{1}g_{2}^{-1} \text{or} \ g_{2}g_{1}^{-1} \in \Gamma$.
\end{definition}

The condition that $\Gamma$ generates $G$ makes the Cayley graph defined above connected. Clearly all Cayley graphs are vertex transitive. 

To fix notations, we first give two examples of Cayley graphs.
\begin{eg}
Let $G=\mathbf{Z}/n\mathbf{Z}$ and $\Gamma=\{\pm 1\}$, $n \in \mathbf{Z}, n\geq 3$. For simplicity, we abuse the notation and denote the Cayley graph $(G, \Gamma)$ by $\mathbf{Z}_n$. The automorphism group of this graph is the dihedral group $D_n$.
\end{eg}

\begin{eg}
Let $G=\mathbf{Z}/n\mathbf{Z}$ and $\Gamma=G-\{0\}$.
Then the Cayley graph $(G, \Gamma)$ is the complete graph $K_n$. The automorphism group of this graph is the symmetric group $S_n$.
\end{eg}

From now on, we concentrate on graph 
bundles over $\mathbf{Z}_n$. We have the following observation which is direct to check.

\begin{proposition}\label{pro1}
    
    Let $\mathbf{Z}_n \times_\phi F$ be a a graph bundle over $\mathbf{Z}_n$. There is an graph automorphism $\rho$ of $\mathbf{Z}_n \times_\phi F$ as follows:

 $$\rho([i],v)=([i+1],\phi_{[i+1][i]}(v))
 $$
where $v\in V_F$.
\end{proposition}

We will use the group action of the cyclic group $\{\rho^{k}, k\in \mathbf{Z}\}$ generated by $\rho$ to explore the properties of the graph bundle $\mathbf{Z}_n \times_\phi F$. Next, we provide an example of vertex transitive non-trivial graph bundle.

%e bundle structures depends only on $$\prod_{l=0}^{n-1}\phi_{[l+1],[l]}.$$

\begin{eg}[\textbf{A non-trivial graph bundle can be vertex transitive}\label{eg:dvb1}]
Let $\mathbf{Z}_n$ be the base graph and $\mathbf{Z}_4$ be the fiber. Define the connection map $\phi$ as
\begin{equation}
\phi_{[i+1][i]}=\begin{cases}
([0][1])([2][3]), &\text{if $[i]=[0]$,}\\
\ id, &\text{otherwise.}
\end{cases}
\end{equation}
where $([0][1])([2][3])$ is the permutation interchanging $[0],[1]$ and also $[2],[3]$. The graph bundle $\mathbf{Z}_n \times_\phi \mathbf{Z}_4$ is non-trivial since there is an unbalanced loop. However, this graph is vertex transitive. In fact we have the following graph automorphism $\tau$:

\begin{equation}
\tau(([i],[j]))=\begin{cases}
([i],[3]) &\text{if $[j]=[0]$,}\\
([i],[2]) &\text{if $[j]=[1]$,}\\
([i],[1]) &\text{if $[j]=[2]$,}\\
([i],[0]) &\text{if $[j]=[3]$,}
\end{cases}
\end{equation}
where $[i]\in \mathbf{Z}/n\mathbf{Z}$ and $[j]\in \mathbf{Z}/4\mathbf{Z}$. It is straightforward to check that the action of the group generated by $\rho$ (defined in Proposition \ref{pro1}) and $\tau$ is transitive. Thus the graph bundle is vertex transitive.
\end{eg}

\begin{remark}
The Cayley graph
$\mathbf{Z}_n$ can be embedded into $S^{1}$  for any $n \geq 2$. Thus both the base and fiber graph of Example \ref{eg:dvb1} can be embedded into $S^{1}$. Moreover, the graph bundle defined above can be embedded into the Klein bottle, which is a non-trivial $S^{1}$ bundle over $S^{1}$.
\end{remark}

Below we present an example of a discrete vector bundle.
\begin{eg}[\textbf{A $K_{m}$ bundle over $\mathbf{Z}_n$}\label{eg2}]

Let $\mathbf{Z}_n$ be the base graph and $K_m$ be the fiber, where $n,m\geq 3$. Define the connection map $\phi$ as
\begin{equation}
\phi_{[i+1][i]}=\begin{cases}
([1][2]...[m-1]), &\text{if $[i]=[0]$,}\\
\ id, &\text{else,}
\end{cases}
\end{equation}
where $([1][2]...[m-1])$ is a cyclic permutation. Note that $\mathbf{Z}_n \times_\phi K_m$ is non-trivial. In fact, it is a non-trivial discrete vector bundle since the connection map fixes $[0] \in K_m$. From Theorem \ref{thm:2}, this graph bundle is not vertex transitive. However, this graph still has a strong symmetry. The group action of $Aut(\mathbf{Z}_n \times_\phi K_m)$ only have two orbits. To see this, one only needs to check the action of the group generated by $\rho$ defined in Proposition\ref{pro1} has two orbits. One of the orbits is $\{([i],[0]), [i]\in \mathbf{Z}_n, [0] \in K_m\}$. Below is a picture of the graph bundle with $m=4,n=4$.
\end{eg}

\begin{figure}[htbp]
  \centering
  \includegraphics[width=0.7\textwidth]{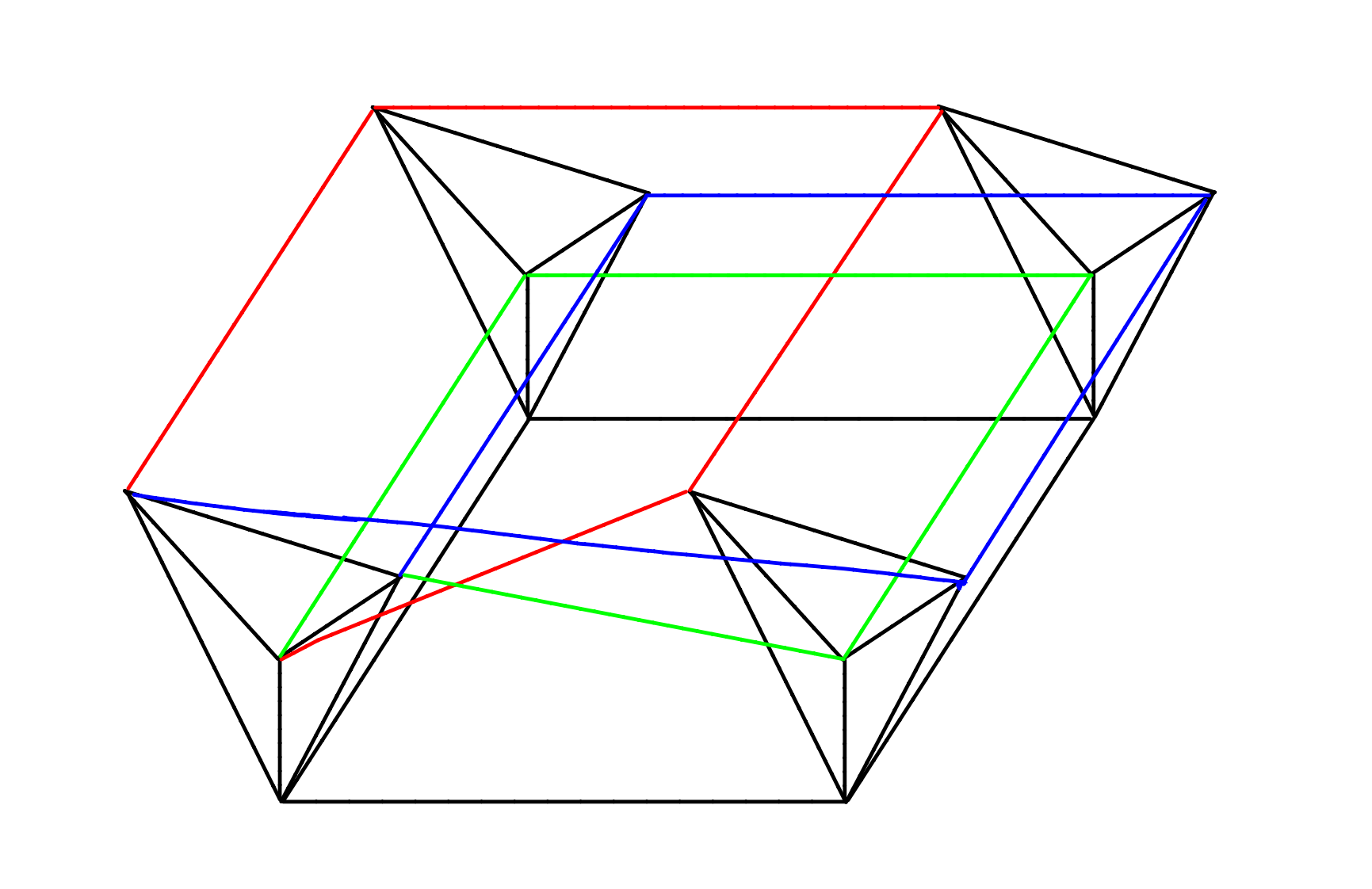}
  \caption{Graph Bundle $\mathbf{Z}_4\times_\phi K_4$ }
\end{figure}

In fact, we can use the same idea as in Example \ref{eg2} to construct a graph bundle with arbitrary given number of orbits under the group action of its automorphic group (as a graph) .

\begin{eg}[\textbf{ $K_{\frac{1}{2}i(i+1)}$ bundles over $\mathbf{Z}_n$ with $i$ orbits}\label{eg3}]

Let $\mathbf{Z}_n$ be the base graph and $K_{m(i)}$ be the fiber, where $n\geq 5$ and $m(i)=\frac{1}{2}i(i+1)$. Define the connection map $\phi$ as
\begin{equation}
\phi_{[i+1][i]}=\begin{cases}([1][2])([3][4][5])([6][7][8][9])...),&\text{if $[i]=[0]$,}\\
\ id, &\text{else,}
\end{cases}
\end{equation}
where $([1][2])([3][4][5])([6][7][8][9])...)$ is the product of cyclic permutations. The graph bundle $\mathbf{Z}_n \times_\phi K_{m(i)}$ is clearly non-trivial. Arguing as in Example \ref{eg2}, there are $i$ orbits under the action of $\{\rho^{k}, k\in \mathbf{Z}\}$, thus at most $i$ orbits under the action of $Aut(\mathbf{Z}_n \times_\phi K_{m(i)})$.

In fact, there are exactly $i$ orbits under the action of $Aut(\mathbf{Z}_n \times_\phi K_{m(i)})$. To see this, we view any loop of length $l$ in a graph $G$ as a graph homomorphism $\beta$ from $\mathbf{Z}_l$ to $G$. We call a loop geodesic-like if, for any $i$, the path \[\{\beta([i]),\beta([i+1]),\beta([i+2])\}\] is the only shortest path from $\beta([i])$ to $\beta([i+2])$. Clearly any automorphism of $G$ maps a geodesic-like loop to another.

From our construction of $\mathbf{Z}_n \times_\phi K_{m(i)}$ it is straightforward to check that for every geodesic-like loop $\{\beta([i])\}$, we have $\beta([i])=\rho (\beta([i\pm 1]))$. So every orbit of $\{\rho^{k}, k \in \mathbf{Z}\}$ can be viewed as an unbased geodesic loop. Since the number of points in each orbit vary there is no automorphism mapping a point from one orbit to another. Hence the number of orbits under the action of $Aut(\mathbf{Z}_n\times_{\phi}K_{m(i)})$ equals $i$.
\end{eg}

\section{Locally Abelian graphs}\label{sec:5}

Now we turn to the topic of Ricci-flatness. Let $G=(V,E)$ be a regular simple graph of degree $d$. Recall G is Ricci flat if for any $x \in V$, there exist maps $\eta_{i}:B_{1}(x)\rightarrow V$, $1\leq i\leq d$, such that\\
(i) $\eta_{i}(u) \sim u$, $u \in B_{1}(x)$,\\
(ii) $\eta_{i}(u)=\eta_{j}(u)$ iff $i=j$,\\
(iii) $\bigcup_{j}\{\eta_{i}(\eta_{j}(x))\}=\bigcup_{i}\{\eta_{j}(\eta_{j}(x))\}$.\\
These $\eta_{i}$ are called local frames of $x$. We call it S-Ricci flat or locally Abelian if, furthermore, \\
(iv)$\eta_{i}(\eta_{j}(x))=\eta_{i}(\eta_{j}(x))$.\\

\noindent A large family of locally Abelian graphs could be constructed via graph bundles. In fact, we have the following result.
\begin{theorem}\label{thm:3}
For any locally Abelian graphs $G$ and $F$ with an associated connection $\phi$ that is balanced on any loop of length 4, the graph bundle $G\times_{\phi}F$ is a locally Abelian graph.
\end{theorem}
\begin{proof}
Fix $(x,v)$ in $G\times_{\phi}F$. Suppose $S_{1}=\{a_{i}\}$ and $S_{2}=\{b_{j}\}$ are local frames of $x$ and $v$ in $G$ and $F$, respectively. We define a local frame of $(x,v)$ in the following way: For every $(s,w) \in B_{1}(x,v)$, define for each $i$ that $$\eta_{a_{i}}(s,w)=(a_{i}(s),\phi(a_{i}(s),s)(w)).$$
For each $j$, the maps $\eta_{b_j}$ are defined by $$\eta_{b_{j}}(x,v)=(x,b_{j}(v)),$$ $$\eta_{b_{j}}(x,b_{k}(v))=(x,b_{k}\circ b_{j}(v)),$$ and $$\eta_{b_{j}}(a_{k}(x),\phi(a_{k}(x),x)(v))=(a_{k}(x), \phi(a_{k}(x),x)(b_{j}(v))).$$ It is direct to check that these local frames satisfy all of (i) to (iv), which implies $G\times_{\phi}F$ is locally Abelian.
\end{proof}

\begin{remark}
    In the previous version, we accidentally omitted the assumption that any loops in the base graph of length 4 is balanced with respect to the connection \(\phi\) in the statement of Theorem \ref{thm:3} and \ref{thm:4}. This assumption was already used unexplicitly in the proof since for 
$$\eta_{a_j}\circ\eta_{a_i}(s,w)=\eta_{a_i}\circ\eta_{a_j}(s,w)$$
to hold for any $(s,w)$, we need
$$\phi(a_j\circ a_i(s),a_i(s))\phi(a_i(s),s)=\phi(a_i\circ a_j(s),a_j(s))\phi(a_j(s),s),$$
which is guaranteed by the assumption.
\end{remark}

We also want to add one more example to clarify Theorem \ref{thm:3} on how to construct an S-Ricci flat graph by graph bundles. The following example is a graph bundle that is not a "locally Catersian product" 
of the base and fiber graph. However, we would like to mention that it is in fact a Cayley graph of an Abelian group which again shows that the discrete vector bundle condition of Theorem \ref{thm:2} is necessary.

\begin{eg}\label{eg:dvb2}

    Let the base graph $G$ be the Cayley graph $$\{\mathbb{Z}^2,\{\pm (1,0), \pm (0,1), \pm (1,1)\}\},$$
    and the fiber graph $F$ be $K_2$, viewed as a Cayley graph of $\mathbb{Z}/2\mathbb{Z}$. Set the connection map $\phi$ on directed edges generated by $\pm (1,1)$ of the base graph as the involution of $K_2$ and the identity map otherwise. The connection is not balanced on $3$-cycles (thus non-trivial) and balanced on all loops of length 4 of the base graph. The frames $\{a_i\}_{i=1}^{6}$ on $G$
    are given naturally such that for all $(m,n)\in \mathbb{Z}^2$,
    $$a_1(m,n)=(m+1,n),\,\,\,\,\,a_2(m,n)=(m-1,n),$$
    $$a_3(m,n)=(m,n+1),\,\,\,\,\,a_4(m,n)=(m,n-1),$$
    $$a_5(m,n)=(m+1,n+1),\,\,\,\,\,a_6(m,n)=(m-1,n-1).$$
 The frame $\{b\}$ on $K_2$ is given such that for all $q \in \mathbb{Z}$,
 $$b\,(q\,\mathrm{mod}\, 2)=q+1\,\mathrm{mod}\, 2.$$
 As in proof of Theorem \ref{thm:3}, for any vertex $(m,n, q\, \mathrm{mod} \,2)$ of $G\times_\phi F$, $m,n,q\in \mathbb{Z}$, by defining frames as
$$\eta_{a_1}(m,n,q\,\mathrm{mod}\, 2)=(m+1,n,q\,\mathrm{mod}\, 2),\,\,\,\,\, \eta_{a_2}(m,n,q\,\mathrm{mod}\, 2)=(m-1,n,q\,\mathrm{mod}\, 2),$$
$$\eta_{a_3}(m,n,q\,\mathrm{mod}\, 2)=(m,n+1,q\,\mathrm{mod}\, 2),\,\,\,\,\, \eta_{a_4}(m,n,q\,\mathrm{mod}\, 2)=(m,n-1,q\,\mathrm{mod}\, 2), $$
$$\eta_{a_5}(m,n,q\,\mathrm{mod}\, 2)=(m+1,n+1,q+1\,\mathrm{mod}\, 2),$$$$\eta_{a_6}(m,n,q\,\mathrm{mod}\, 2)=(m-1,n-1,q-1\,\mathrm{mod}\, 2), $$
$$\eta_{b}(m,n,q\,\mathrm{mod}\, 2)=(m,n,q+1\,\mathrm{mod}\, 2),$$
we see that it is S-Ricci flat. In fact, it is isometric to the Cayley graph $$\{\mathbb{Z}^2\times \mathbb{Z}/2\mathbb{Z},\{\pm(1,0,0\,\mathrm{mod}\, 2),\pm(0,1,0\,\mathrm{mod}\, 2),\pm(1,1,1\,\mathrm{mod}\, 2),(0,0,1\,\mathrm{mod}\, 2)\}\}$$
and each frame $\eta$ of the graph bundle corresponds to a generator of this Cayley graph.
\end{eg}

All Abelian Cayley graphs are locally Abelian. It turns out that the set of Abelian Cayley graphs is a small subset of locally Abelian graphs.
\begin{theorem}\label{thm:4}
For any nontrivial discrete vector bundle $G\times_{\phi}F$ with $G$ and $F$ being Abelian Cayley graphs and $\phi$ being balanced on any loop of length 4, $G\times_{\phi}F$ is locally Abelian rather than an Abelian Cayley graph.
\end{theorem}

\begin{proof}
The graph bundle $G\times_{\phi}F$ is locally Abelian due to Theorem 3. It is not a Cayley graph since it is not vertex transitive by Theorem \ref{thm:2}.
%From Theorem \ref{thm:2} we obtain immediately the following result.
\end{proof}

\section*{Acknowledgement}
We thank Xiangru Zeng for bringing our attention to the pull-back of the tangent bundle $TS^2$ under the projection map $\pi: S^{2} \times \mathbf{R} \to  S^{2}$. WL wants to thank Jingbin Cai for discussions. We are very grateful to the anonymous referee for comments and suggestions which have contributed  a lot to improve the quality of our article. We also thank Florentin Münch for pointing our error in the previous version out and also Norbert Peyerimhoff, Supanat Kamtue and Joe Thomas for further discussions.
This work is supported by the National Key R and D Program of China 2020YFA0713100, the National Natural Science Foundation of China (No. 12031017), and Innovation Program for Quantum Science and Technology 2021ZD0302902.

\bibliographystyle{plainnat}

\begin{thebibliography}{99}

\bibitem{BE1985}
Dominique Bakry and Michel \'{E}mery.
\newblock Diffusions hypercontractives.
\newblock In {\em S\'{e}minaire de probabilit\'{e}s, {XIX}, 1983/84}, volume
  1123 of {\em Lecture Notes in Mathematics}, pages 177--206. Springer, Berlin, 1985.

\bibitem{banivct2009edge}
Iztok Bani{\v{c}}, Rija Erve{\v{s}}, and Janez {\v{Z}}erovnik.
\newblock The edge fault-diameter of Cartesian graph bundles.
\newblock {\em European Journal of Combinatorics}, 30(5):1054--1061, 2009.

\bibitem{banivc2006fault}
Iztok Bani{\v{c}} and Janez {\v{Z}}erovnik.
\newblock Fault-diameter of Cartesian graph bundles.
\newblock {\em Information Processing Letters}, 100(2):47--51, 2006.

\bibitem{banivct2010wide}
Iztok Bani{\v{c}} and Janez {\v{Z}}erovnik.
\newblock Wide diameter of Cartesian graph bundles.
\newblock {\em Discrete Mathematics}, 310(12):1697--1701, 2010.

\bibitem{BHJLW17survey}
Frank Bauer, Bobo Hua, J\"{u}rgen Jost, Shiping Liu, and Guofang Wang.
\newblock The geometric meaning of curvature: local and nonlocal aspects of
  {R}icci curvature.
\newblock In {\em Modern approaches to discrete curvature}, volume 2184 of {\em
  Lecture Notes in Mathematics}, pages 1--62. Springer, Cham, 2017.

\bibitem{chae1993characterlstlc}
Younki Chae, Jin~Ho Kwak, and Jaeun Lee.
\newblock Characterlstlc polynomials of some graph bundles.
\newblock {\em Journal of the Korean Mathematical Society}, 30(1):229--249,
  1993.

\bibitem{chung1996logarithmic}
Fan-Rong King Chung and Shing-Tung Yau.
\newblock Logarithmic harnack inequalities.
\newblock {\em Mathematical Research Letters}, 3(6):793--812, 1996.

\bibitem{cushing2021curvatures}
David Cushing, Supanat Kamtue, Riikka Kangaslampi, Shiping Liu, and Norbert
  Peyerimhoff.
\newblock Curvatures, graph products and {R}icci flatness.
\newblock {\em Journal of Graph Theory}, 96(4):522--553, 2021.

\bibitem{Elworthy1991}
K.~D. Elworthy.
\newblock Manifolds and graphs with mostly positive curvatures.
\newblock In {\em Stochastic Analysis and Applications ({L}isbon, 1989)},
  volume~26 of {\em Progress in  Probability}, pages 96--110. Birkh\"{a}user Boston,
  Boston, MA, 1991.

\bibitem{ervevs2013mixed}
Rija Erve{\v{s}} and Janez {\v{Z}}erovnik.
\newblock Mixed fault diameter of Cartesian graph bundles.
\newblock {\em Discrete Applied Mathematics}, 161(12):1726--1733, 2013.

\bibitem{feng2006zeta}
Rongquan Feng and Jin-Ho Kwak.
\newblock Zeta functions of graph bundles.
\newblock {\em Journal of the Korean Mathematical Society}, 43(6):1269--1287,
  2006.

\bibitem{hong1999bipartite}
Sungpyo Hong, Jin~Ho Kwak, and Jaeun Lee.
\newblock Bipartite graph bundles with connected fibres.
\newblock {\em Bulletin of the Australian Mathematical Society},
  59(1):153--161, 1999.

\bibitem{HuaLin2016survey}
Bobo Hua and Yong Lin.
\newblock Curvature notions on graphs.
\newblock {\em Frontiers of Mathematics in China}, 11(5):1275--1290, 2016.

\bibitem{imrich1997recognizing}
Wilfried Imrich, Toma{\v{z}} Pisanski, and Janez {\v{Z}}erovnik.
\newblock Recognizing Cartesian graph bundles.
\newblock {\em Discrete Mathematics}, 167:393--403, 1997.

\bibitem{kamtue2018curvature}
Supanat Kamtue.
\newblock Combinatorial, {B}akry-Émery, {O}llivier's {R}icci curvature notions
  and their motivation from riemannian geometry.
\newblock {\em arXiv:1803.08898}, 2018.

\bibitem{kim2008generalized}
Dongseok Kim, Hye~Kyung Kim, and Jaeun Lee.
\newblock Generalized characteristic polynomials of graph bundles.
\newblock {\em Linear Algebra and its Applications}, 429(4):688--697, 2008.

\bibitem{klavzar1995chromatic}
Sandi Klav{\v{z}}ar and Bojan Mohar.
\newblock The chromatic numbers of graph bundles over cycles.
\newblock {\em Discrete Mathematics}, 138(1-3):301--314, 1995.

\bibitem{klavvzar1995coloring}
Sandi Klav{\v{z}}ar and Bojan Mohar.
\newblock Coloring graph bundles.
\newblock {\em Journal of Graph Theory}, 19(2):145--155, 1995.

\bibitem{kwak2001characteristic}
Jin~Ho Kwak and Young~Soo Kwon.
\newblock Characteristic polynomials of graph bundles having voltages in a
  dihedral group.
\newblock {\em Linear Algebra and its Applications}, 336(1-3):99--118, 2001.

\bibitem{kwak1990isomorphism}
Jin~Ho Kwak and Jaeun Lee.
\newblock Isomorphism classes of graph bundles.
\newblock {\em Canadian journal of mathematics}, 42(4):747--761, 1990.

\bibitem{kwak1996isoperimetric}
Jin~Ho Kwak, Jaeun Lee, and Moo~Young Sohn.
\newblock Isoperimetric numbers of graph bundles.
\newblock {\em Graphs and Combinatorics}, 12(3):239--251, 1996.

\bibitem{LLY2011Tohoku}
Yong Lin, Linyuan Lu, and Shing-Tung Yau.
\newblock Ricci curvature of graphs.
\newblock {\em Tohoku Mathematical Journal, Second Series}, 63(4):605--627, 2011.

\bibitem{LinYau2010}
Yong Lin and Shing-Tung Yau.
\newblock Ricci curvature and eigenvalue estimate on locally finite graphs.
\newblock {\em Mathematical Research Letter}, 17(2):343--356, 2010.

\bibitem{Maas17survey}
Jan Maas.
\newblock Entropic {R}icci curvature for discrete spaces.
\newblock In {\em Modern Approaches to Discrete Curvature}, volume 2184 of {\em
  Lecture Notes in Mathematics}, pages 159--174. Springer, Cham, 2017.

\bibitem{mohar1988maximum}
Bojan Mohar, Toma{\v{z}} Pisanski, and Martin {\v{S}}koviera.
\newblock The maximum genus of graph bundles.
\newblock {\em European Journal of Combinatorics}, 9(3):215--224, 1988.

\bibitem{MW2019Adv}
Florentin M\"{u}nch and Rados\l aw~K. Wojciechowski.
\newblock Ollivier {R}icci curvature for general graph {L}aplacians: heat
  equation, {L}aplacian comparison, non-explosion and diameter bounds.
\newblock {\em Advances in Mathematics}, 356:106759, 45, 2019.

\bibitem{Ollivier2009JFA}
Yann Ollivier.
\newblock Ricci curvature of {M}arkov chains on metric spaces.
\newblock {\em Journal of Functional Analysis}, 256(3):810--864, 2009.

\bibitem{pisanski1983edge}
Toma{\v{z}} Pisanski, John Shawe-Taylor, and Joz{\v{e}} Vrabec.
\newblock Edge-colorability of graph bundles.
\newblock {\em Journal of Combinatorial Theory, Series B}, 35(1):12--19, 1983.

\bibitem{pisanski2009hamilton}
Toma{\v{z}} Pisanski and Janez {\v{Z}}erovnik.
\newblock Hamilton cycles in graph bundles over a cycle with tree as a fibre.
\newblock {\em Discrete Mathematics}, 309(17):5432--5436, 2009.

\bibitem{Schmuckenschlaeger1999}
Michael Schmuckenschl\"{a}ger.
\newblock Curvature of nonlocal {M}arkov generators.
\newblock In {\em Convex Geometric Analysis ({B}erkeley, {CA}, 1996)},
  volume~34 of {\em Mathematical Sciences Research Institute Publications}, pages 189--197. Cambridge
  Univ. Press, Cambridge, 1999.

\bibitem{sohn1994characteristic}
Moo~Young Sohn and Jaeun Lee.
\newblock Characteristic polynomials of some weighted graph bundles and its
  application to links.
\newblock {\em International Journal of Mathematics and Mathematical Sciences},
  17(3):503--510, 1994.

\bibitem{vzerovnik2000recognition}
Janez {\v{Z}}erovnik.
\newblock On recognition of strong graph bundles.
\newblock {\em Mathematica Slovaca}, 50(3):289--301, 2000.

\bibitem{zmazek2000recognizing}
Blaz Zmazek and Janez Zerovnik.
\newblock Recognizing weighted directed Cartesian graph bundles.
\newblock {\em Discussiones Mathematicae Graph Theory}, 20(1):39--56, 2000.

\bibitem{zmazek2002algorithm}
Bla{\v{z}} Zmazek and Janez {\v{Z}}erovnik.
\newblock Algorithm for recognizing Cartesian graph bundles.
\newblock {\em Discrete Applied Mathematics}, 120(1-3):275--302, 2002.

\bibitem{zmazek2002unique}
Bla{\v{z}} Zmazek and Janez {\v{Z}}erovnik.
\newblock Unique square property and fundamental factorizations of graph
  bundles.
\newblock {\em Discrete Mathematics}, 244(1-3):551--561, 2002.

\bibitem{zmazek2006domination}
Blaz Zmazek and Janez Zerovnik.
\newblock On domination numbers of graph bundles.
\newblock {\em Journal of Applied Mathematics and Computing}, 22:39--48, 2006.

\end{thebibliography}

\end{document}